\documentclass[a4paper,reqno]{amsart}
\usepackage{pgf,tikz,pgfplots}
\usepackage{amsmath}
\usepackage{paralist}
\usepackage{graphics}
\usepackage{epsfig}
\usepackage{amsfonts}
\usepackage{amssymb}
\usepackage{hyperref,cite}

\usepackage{amsthm}
\usepackage{enumerate}
\usepackage[mathscr]{eucal}
\usepackage{eqlist}

\def\teich{{Teichm\"{u}ller}}
\def\mobi{{M{\"{o}}bius}}

\newtheorem{theorem}{Theorem}[section]
\newtheorem{lemma}[theorem]{Lemma}

\def\ifl{\iffalse }

\numberwithin{equation}{section}

\numberwithin{equation}{section}

\theoremstyle{remark}

\begin{document}
\baselineskip=1.3\baselineskip
\title[ On Carleson measure]
{On Carleson measures induced by Beltrami coefficients being compatible with Fuchsian groups}

\author{Huo Shengjin}
\address{Department of Mathematics, Tianjin Polytechnic University, Tianjin 300387, China} \email{huoshengjin@tjpu.edu.cn}
%
%
%


\thanks{This work was supported by the Science and Technology Development Fund of Tianjin Commission for Higher Education(Grant No.2017KJ095) and by the National Natural Science Foundation of China (Grant No.11401432..}
%
\begin{abstract}
 Suppose $\mu$ be a Beltrami coefficient on the unit disk,  which is compatible with a convex co-compact Fuchsian group $G$ of the second kind. In this paper we show that if
$\displaystyle\frac{|\mu|^{2}}{1-|z|^{2}}dxdy $ satisfies the Carleson condition on the infinite boundary boundary of the Dirichlet domain of $G$,
then $\displaystyle\frac{|\mu|^{2}}{1-|z|^{2}}dxdy$  is a Carleson measure on the unit disk.

\end{abstract}

\maketitle

\section{1 Introduction }

A Fuchsian group is a discrete {\mobi} group $G$ acting on the unit disk $\Delta$. A Fuchsian group is called the first kind if the limit is the entire circle and is called the second kind otherwise. A Fuchsian group $G$ is called cocompact if $\Delta/G$ is compact and is called convex cocompact if $G$ is finitely generated without parabolic elements. All cocompact groups are the first kind and convex cocompact groups minus cocompact groups are the second kind. A Fuchsian  group $G$ is of divergence type if $$\Sigma_{g\in G}(1-|g(0)|)=\infty~~\text{or}~~\sum_{g\in G}\exp(-\rho(0,g(0)))=\infty,$$
where $\rho(0,g(0))$ is the hyperbolic distance between $0$ and $g(0)$. Otherwise, we say that it is of convergence type.
All the second kind groups are of convergence type. For more details about Fuchsian groups, see \cite{Da}.

For $g$ in $G$, we denote by $\mathcal{D}_{z}(g)$ the closed hyperbolic half-plane containing $z$, bounded by the perpendicular bisector of the segment $[z, g(z)]_{h}$. The Dirichlet fundamental domain $\mathcal{F}_{z}(G)$ of $G$ centered at $z$ is the intersection of all the sets $\mathcal{D}_{z}(g)$ with $g$ in $G-\{id\}$. For simplicity, in this paper we use the notion $\mathcal{F}$ for the Dirichlet domain $\mathcal{F}_{z}(G)$ of $G$ centered at $z=0.$

A positive measure $\lambda$ defined in a simply connected domain $\Omega$ is called a Carleson measure if there exists some constant C which is independent  of $r$ such that, for all $0<r<diameter(\partial\Omega)$ and $z\in \partial\Omega$,
 $$\lambda(\Omega\cap D(z,r))\leq Cr. $$
 The infimum of all such $C$ is called the Carleson norm of $\lambda,$ denoted by
 $\parallel\lambda\parallel_{*}.$ Let $\Delta$ be the unit disk. In this paper, we mainly focus our attention on the case $\Omega=\Delta.$ We denote by $CM(\Delta)$ the set of all Carleson measures on $\Delta.$

For a positive measure $|\mu(z)|dxdy$, we say $\mu\in CM^{*}(\Delta)$ if and only if the measure $$\displaystyle\frac{|\mu|^{2}}{1-|z|^{2}}dxdy\in CM(\Delta).$$

The importance of the class $CM^{*}(\Delta)$ lies in the fact that it plays a crucial role in the theory of BMOA-{\teich} space, see \cite{AZ,Bi2, Cu, SW} etc.

For a Fuchsian group $G$, suppose $\mu(z)$ is a bounded measurable function on $\Delta$ which satisfies $$||\mu(z)||_{\infty}<1 ~~\text {and} ~~\mu(z)=\mu(g(z))\overline{g'(z)}/g'(z)$$ for every $g\in G$, then we say $\mu$ is a  $G$-compatible Beltrami coefficient (or complex dilatation). We denote by $M(G)$ the set of all $G$-compatible Beltrami coefficients. For a $G$-compatible Beltrami coefficient $\mu$, if the measure
$$\displaystyle\frac{|\mu|^{2}}{1-|z|^{2}}dxdy$$
is a Carleson measure on $\Delta,$ when the Carleson norm is small, then $f_{\mu}(\partial\Delta)$ is a rectifiable (chord-arc) curve, where $f_{\mu}$ is the quasiconformal mapping of the complex plane $\mathbb{C}$ with $i$, $1$ and $-i$ fixed, whose Beltrami coefficient equals to $\mu$ a.e. on the unit disk and equals to zero on the outside of the unit disk. This is essential for the proof of the convergence-type first-kind Fuchsian groups failing to have Bowen's property, see \cite{AZ1}. It is also the critical method to prove that the convergence-type Fuchsian groups  fail to have Ruelle's property, see \cite{HZ, HW}.

It is important to investigate in which condition the $G$-compatible Beltrami coefficients belong to $CM^{*}(\Delta).$ We call the intersection of $\overline{\mathcal{F}}$ with the unit circle $\partial\Delta$ the boundary at infinity of $\mathcal{F}$, denoted by $\mathcal{F}(\infty)$. In this paper, we have

\begin{theorem}\label{main}
Let $G$ be a convex cocompact Fuchsian group of the second kind  and $\mathcal{F}$ the Dirichlet  domain of $G$ centered at $0$. For $\mu\in M(G)$, if there exists a constant $C$ such that, for any $\xi\in \mathcal{F}(\infty)$(i.e. $\xi$ is in the free edges of $\mathcal{F}$) and for any $0<r<1$,
$$\iint_{B(\xi,r)}\displaystyle\frac{|\mu|^{2}\chi_{\mathcal{F}}}{1-|z|^{2}}dxdy \leq Cr.$$
Then $\mu$ is in $CM^{*}(\Delta),$ where $\chi_{\mathcal{F}}$ is the characteristic function of the Dirichlet domain $\mathcal{F}.$
\end{theorem}
Notice that Theorem \ref{main} fails for the case of convex cocompact groups of the first kind (i.e. cocompact groups), since  Bowen \cite{Bo} showed that cocompact groups hold a rigidity property, now called Bowen's property, i.e. the image of the unit circle under any quasiconformal map whose Beltrami coefficient compatible with a cocompact group,  is either a circle or  has Hausdorff dimension bigger than 1. Hence for any $\mu$ being compatible with cocompact groups, the measure $\displaystyle \frac{|\mu|^{2}}{1-|z|^{2}}dxdy$ is not a Carleson measure.


By this theorem we see that the Carleson property of the measures which are compatible with the convex compact second-kind Fuchsian groups can be checked from the points  in the set $\mathcal{F}(\infty)$ i.e., the boundary at infinity of the Dirichlet domain $\mathcal{F}$.

\medskip
\noindent{\bf Notation.} In this paper $\chi_{A}$ always denotes the characteristic function of the domain $A.$

\section{ Some lemmas}

 The following lemma will be used several times in this paper. I give a short proof here.

\begin{lemma}\label{le}
  Let $\mu$ being a essentially bounded measurable function on $\Delta$. If the measure $\displaystyle\frac{|\mu|^{2}}{1-|z|^{2}}dxdy$ is in $CM(\Delta)$, then there exists a constant $C$  such that, for any $\xi\in\overline{\Delta} $ and all $0<r<2$,
$$\iint_{B(\xi,r)\cap\Delta}\displaystyle\frac{|\mu|^{2}}{1-|z|^{2}}dxdy \leq Cr,$$
where the constant $C$ depends only on the Carleson norm of the measure $\displaystyle\frac{|\mu|^{2}}{1-|z|^{2}}dxdy$ and the essential norm of $\mu$.
\end{lemma}

\begin{proof}
We first choose $0<r<2$ and fix it. For any $\xi\in \overline{\Delta}$, if $\xi\in\partial\Delta$, there is nothing need to prove. We suppose
 $\xi\in \Delta.$ If  $dist(\xi, \partial\Delta)\geq 2r$ (this case only happens when $0<r<0.5$), where $dist(\cdot,\cdot)$ denotes the Euclidean distance. Then we have
 $$\iint_{B(\xi, r)}\displaystyle\frac{|\mu|^{2}}{1-|z|^{2}}dxdy\leq \displaystyle\frac{||\mu||_{\infty}\pi r^{2}}{1-|1-r|^{2}}
 =\displaystyle\frac{||\mu||_{\infty}\pi r}{2-r}\leq \pi||\mu||_{\infty} r.$$

 For the case $dist(\xi, \partial\Delta)\leq 2r,$ we can choose a point $\eta\in\partial\Delta$ such that $dist (\eta,\xi)<2r.$ Then we have
 $B(\xi,r)\subset B(\eta, 4r)$ and
 \begin{equation}
 \iint_{B(\xi, r)\cap \Delta}\displaystyle\frac{|\mu|^{2}}{1-|z|^{2}}dxdy\leq
 \iint_{B(\eta, 4r)\cap\Delta}\displaystyle\frac{|\mu|^{2}}{1-|z|^{2}}dxdy\leq 4C^{*}r,
 \end{equation}
 where $C^{*}$ is the Carleson norm of the measure $\displaystyle\frac{|\mu|^{2}}{1-|z|^{2}}dxdy.$

Hence we let $C=\max\{\pi||\mu||_{\infty}, 4C^{*}\}$ and the lemma follows.
\end{proof}

\noindent{\bf Remark.} By this lemma we see that for any simply connected domain $\Omega\subset \Delta$, If $\displaystyle\frac{|\mu|^{2}}{1-|z|^{2}}dxdy$ is a Carleson measure on $\Delta$, then it is also a Carleson measure on $\Omega.$

In order to prove Theorem \ref{main}, we will  need the following lemma which
 essentially belongs to Astala and Zinsmeister, see \cite{AZ}, or \cite{AZ1}.
\begin{lemma}\label{le1}
For a convergence-type Fuchsian group $G$ and $\mu$ in $M(G)$, if there exists  a $0<t<1$ such that the support set of $\mu\chi_{\mathcal{F}}$ is contained in the ball $B(0, t)$ with center $0$ and radius $t.$ Then $\mu$ is in $CM^{*}(\Delta).$
\end{lemma}

 For the readers to see more clearly about the property of $\mu,$ we give the detail of proof of this lemma here.

\begin{proof}
  Recall that a sequence $\{z_{j}\}$ is called an interpolating sequence of $\Delta$ if $$(i) ~~\exists \delta>0, ~\rho(z_{j}, z_{k})\geq\delta~~if ~~j\neq k; $$
 $$(ii)~~\sum (1-|z_{i}|^{2})\delta_{z_{i}}\in CM(\Delta),$$
 where $\delta_{z}$ stands for the Dirac mass at $z$.

We first show that the sequence $\{g(0)\}_{g\in G}$
is an interpolating sequence of the unit disk $\Delta$.

The sequence $\{g\}_{g\in G}$ satisfies the property $(i)$ of the interpolating sequence  immediately from the action of Fuchsian group being discrete.

For the property $(ii)$, by a result due to Carleson \cite{Ca}, we know that
$$\sum_{g\in G} (1-|g|^{2})\delta_{g}\in CM(\Delta) \eqno(3.1)$$
is equivalent to the Blaschke product $$\inf_{g_{i}}\prod_{g_{j}\in G, g_{j}\neq g_{i}}\big |\displaystyle\frac{g_{i}-g_{j}}{1-\bar{g_{i}}g_{j}}\big|
\geq\delta>0,~~j=1,2,\cdot\cdot\cdot. \eqno(3.2)$$
 In order to show (3.2), it is enough to prove that for any $g_{i}\neq g_{k},$
$$\prod_{g_{j}\in G, g_{j}\neq g_{i}}\big |\displaystyle\frac{g_{i}-g_{j}}{1-\bar{g_{i}}g_{j}}\big|\equiv\prod_{g_{j}, g_{k}\in G,j\neq k}\big |\displaystyle\frac{g_{k}-g_{j}}{1-\bar{g_{k}}g_{j}}\big|,~~j=1,2,\cdot\cdot\cdot. \eqno(3.3)$$
Note that
$$\big |\displaystyle\frac{g_{i}-g_{j}}{1-\bar{g_{i}}g_{j}}\big|=\tanh 2\rho(g_{i},g_{j}),$$
where $\rho(g_{i},g_{j})$ denotes the hyperbolic distance between $g_{i}$ and $g_{j}.$
Similarly£¬
$$\big |\displaystyle\frac{g_{k}-g_{j}}{1-\bar{g_{k}}g_{j}}\big|=\tanh 2\rho(g_{k},g_{j}).$$
Let $\gamma=g_{k}\circ g_{i}^{-1},$ we have $g_{k}=\gamma\circ g_{i}$ and
\begin{eqnarray*}
&&\prod_{g_{j}\in G, g_{j}\neq g_{k}}\big |\displaystyle\frac{g_{k}-g_{j}}{1-\bar{g_{k}}g_{j}}\big|=
\prod_{g_{j}\in G, g_{j}\neq g_{k}}\tanh(2\rho(g_{k}, g_{j}))\\
&=&
\prod_{g_{j}\in G, g_{j}\neq g_{k}}\tanh(2\rho(\gamma\circ g_{i}, g_{j}))
=\prod_{g_{j}\in G,g_{j}\neq g_{k}}\tanh(2\rho( g_{i}, \gamma^{-1}\circ g_{j}))\\
&=&\prod_{ g_{j}\in G,g_{j}\neq g_{i}}\big |\displaystyle\frac{g_{i}-g_{j}}{1-\bar{g_{i}}g_{j}}\big|.
\end{eqnarray*}

Let $g_{i}(0)=0$ and in this case
\begin{eqnarray*}
&&\prod_{g_{j}\in G,g_{j}\neq g_{i}}\big |\displaystyle\frac{g_{i}-g_{j}}{1-\bar{g_{i}}g_{j}}\big|=
\prod_{g_{j}\neq 0} |g_{j}|\\
&=&\exp{(\sum_{g_{j}\neq 0} \ln|g_{j}|)}\geq\exp {(C\sum_{g_{j}\neq 0}(1-|g_{j}|))},
\end{eqnarray*}
where $C$ is some universal constant.

By the definition of the convergence type group we know that the sequence $\{g\}_{g\in G}$ is an interpolating sequence.

We continue to prove this lemma.
Suppose the support set of $\mu\chi_{\mathcal{F}}$, denoted by $Supp(\mu_{\mathcal{F}})$ which is contained in the ball $B(0,t)$.
For any $\xi\in \partial\Delta$ and $0<r\leq 2$, we have
\begin{eqnarray*}
&&\iint_{\Delta\cap B(\xi,r)}\displaystyle\frac{|\mu(z)|^{2}}{1-|z|^{2}}dxdy\\
&=&\sum_{g\in G}\iint_{g(B(0,t))\cap B(\xi,r)}\displaystyle\frac{|\mu(z)|^{2}}{1-|z|^{2}}dxdy\\
&=&\sum_{g\in G}\iint_{g(B(0,t))}\displaystyle\frac{|\mu(z)|^{2}}{1-|z|^{2}}\chi_{B(\xi,r)}dxdy\\
&\leq&\sum_{g\in G}\parallel\mu\parallel_{\infty}^{2}\iint_{g(B(0,t))}
\displaystyle\frac{1}{1-|z|^{2}}\chi_{B(\xi,r)}dxdy.
\end{eqnarray*}

It is easy to see that the hyperbolic radius $t_{\rho}$ of the Euclidean disk $B(0,t)$ is $\ln\displaystyle\frac{1+t}{1-t}$ . Hence for any $g\in G$, the disk $g(B(0,t))$ is a hyperbolic disk with center $g(0)$ and hyperbolic radius $t_{\rho}.$
By some simple calculation or by \cite{Be} we know that the disk $g(B(0,t))$ is contained in the Euclidean disk $B(g(0), R_{g}) $, where the radius $R_{g}$  is equal to $$\displaystyle\frac{(1+|g(0)|)(1-e^{t_{\rho}})(1-|g(0)|)}{(1+|g(0)|)+e^{t_{\rho}}(1-|g(0)|)}\leq C(1-|g(0)|),$$
where C is some constant depending only on $t$.

Combine the above discussion we have
\begin{align*}
&\iint_{\Delta\cap B(\xi,r)}\displaystyle\frac{|\mu(z)|^{2}}{1-|z|^{2}}dxdy\\
&\leq
\sum_{g(0)\in B(\xi, r)}\displaystyle\frac{||\mu||_{\infty}\pi R_{g}^{2}}{1-|1-R_{g}|^{2}}\\
&\leq C'\sum_{g(0)\in B(\xi, r)}(1-|g(0)|)\leq C^{*}r,
\end{align*}
where the constant $C^{*}$ depends only on $C'$ and the Carleson norm of the measure $\sum_{g\in G}(1-|g(0)|)\delta_{g(0)}.$
 Hence the  lemma holds.
\end{proof}

\noindent{\bf Remark:} In \cite{Bi2}, Bishop used the norm property of Schwarzian derivative of holomorphic function under hyperbolic metric to give another proof of Lemma \ref{le1} for the case the Beltrami coefficient $\mu $ supported on a compact subset of the surface $\Delta/G.$

A Jordan curve $\gamma$ is said to be a chord-arc curve if there exists a constant $C$ such that for any two points $\xi_{1}$, $\xi_{2}\in\gamma$, the length of the arc $\gamma_{\xi_{1},\xi_{2}}$  satisfies
$$length(\gamma_{\xi_{1},\xi_{2}})\leq Cd(\xi_{1}, \xi_{2}),$$
where $\gamma_{\xi_{1},\xi_{2}}$ is the shorter arc of $\gamma$ with endpoints
$\xi_{1}, \xi_{2}$ and $d(\xi_{1}, \xi_{2})$ means the Euclidean distance between $\xi_{1}$ and $\xi_{2}$.

A result from \cite{Z} says that

\begin{lemma}\label{le2}(\cite{Z})
Let $\Omega$ be a chord-arc domain. Then the following are equivalent:

(a) $d\nu$ is a Carleson measure for $\Omega.$

(b) For $0<p<\infty,$ and $f\in H^{p}(\Omega),$
$$\iint_{\Omega}|f|^{p}dv\leq C\int_{\partial \Omega}|f|^{p}ds,$$
where $H^{p}(\Omega)=\{f: f \text{is analytic on }\Omega \text{ and } \int_{\partial \Omega}|f|^{p}ds<\infty \}$ and the constant $C$ depends only on the the Carleson norm of $d\nu.$
\end{lemma}

\noindent{\bf Remark.} Lemma \ref{le2} was first given by Carleson [\cite{Ga},Theorem 3.9, P.{61}] when $\Omega$ is the upper half plane. Zinsmeister proved that Carleson's theorem remains true for chord-arc domains, see \cite{Z}.

By the preparatory work we have done, it is time to give the proof of Theorem {\ref{main}.}
\section{Proof of Theorem \ref{main}}
\begin{proof}
Let $G$ be a second-kind convex cocompact group and $\mathcal{F}$ be the Dirichlet domain of $G$ with center $0.$  Let $\mu$ be an element in $M(G).$
The intersection of the closure of $\mathcal{F}$ with $\partial\Delta$ contains finitely many intervals which are called free edges of $\mathcal{F}$, denoted by $I_{1},$ $I_{2},$ $\cdot\cdot\cdot$$I_{n}.$

For any $1\leq i\leq n,$ let $q_{i,1}, q_{i,2}$ be the endpoints of $I_{i}$. It is well known that both $q_{i,1}, q_{i,2}$ do not belong to the limit set. Both sides of $q_{i,j}$ ($j=1,$ or $2$) are free sides of Dirichlet domains with different centers.

By the statement of the theorem we know  there exists a constant $C$ such that for any $1\leq i\leq n,$ we can choose a ball $B_{i}$ such that $B_{i}\cap \partial\Delta$ contains no limit points of $G$ and $I_{i}\subset B_{i}\cap \partial\Delta$ and for any point $\xi\in I_{i}$ and $0<r<2,$
$$\iint_{B(\xi, r)\cap\Delta}\displaystyle
 \frac{|\mu(z)|^{2}}{1-|z|^{2}}\chi_{B_{i}\cap\Delta}dxdy\leq Cr,$$
 furthermore, the set $\overline{\mathcal{F}}-\bigcup_{i=1}^{n}(B_{i}\cap \mathcal{F})$ is compact, denoted by $\mathcal{F}_{c}.$

 By Lemma \ref{le} we know that
 the measure $$\displaystyle
 \frac{|\mu(z)|^{2}}{1-|z|^{2}}dxdy$$
  is a Carleson measure
 on the domain $B_{i}\cap \mathcal{F}.$

 We divide  $\mu$ into two parts.   Let
  $$\mu=\sum_{g\in G}\mu\chi_{g(\mathcal{F}_{c})}+\sum_{g\in G}\mu\chi_{g(B)},$$
 where
 $B=\bigcup_{i=1}^{n}(B_{i}\cap \mathcal{F}).$

By Lemma\ref{le1}, we know that the measure $\sum_{g\in G}\mu\chi_{g(\mathcal{F}_{c})}$ is a Carleson measure on $\Delta.$ In the following we only need to show that $\sum_{g\in G}\mu\chi_{B}$
is also a Carleson measure. For the simplified the notion, we may suppose
$\mu=\sum_{g\in G}\mu\chi_{B}.$

Let $\xi$ be an arbitrary point of $\partial\Delta$ and  $r$ a positive real number less than $2$. In the following we will find a positive constant $C^{*}$ which does not depend on $\xi$ and $r$ such that
$$\iint_{B(\xi,r)\cap\Delta}\displaystyle\frac{|\mu|^{2}}{1-|z|^{2}}dxdy \leq C^{*}r. \eqno(3.1)$$

We first consider the following special case. If there exists $g\in G$ such that
$g(B(\xi,r)\cap\Delta)\subset \mathcal{F}.$ By Lemma \ref{le} we know that $\displaystyle\frac{|\mu|^{2}}{1-|z|^{2}}dxdy$  is a Carleson measure on the domain $g(B(\xi,r)\cap\Delta)$. Then we have
\begin{equation*}
\begin{aligned}
 \iint_{B(\xi, r)\cap \Delta}\displaystyle\frac{|\mu(w)|^{2}}{1-|w|^{2}}dudv&\leq
 \iint_{g(B(\xi, r)\cap\Delta)}\displaystyle
 \frac{|\mu(g^{-1}(z))|^{2}}{1-|g^{-1}(z)|^{2}}|(g^{-1})'(z)|^{2}dxdy\\
 &= \iint_{g(B(\xi, r)\cap\Delta)}\displaystyle
 \frac{|\mu(g^{-1}(z))\frac{\overline{(g^{-1})'(z)}}
 {(g^{-1})'(z)}|^{2}}{1-|g^{-1}(z)|^{2}}|(g^{-1})'(z)|^{2}dxdy\\
 &=\iint_{g(B(\xi, r)\cap\Delta)}\displaystyle
 \frac{|\mu(z)|^{2}}{1-|z|^{2}}|(g^{-1})'(z)|dxdy.
 \end{aligned}
 \label{3.2}
 \end{equation*}
Since $g$ is a {\mobi} transformation, the domain $g(B(\xi,r)\cap\Delta)$ is a chord-arc domain. By Lemma \ref{le2}, we have
\begin{align*}
&\iint_{g(B(\xi, r)\cap\Delta)}\displaystyle
 \frac{|\mu(z)|^{2}}{1-|z|^{2}}|(g^{-1})'(z)|dxdy\\
 &\leq C_{1}\int_{\partial g(B(\xi,r)\cap\Delta)}|(g^{-1})'(z)|ds\\
 &= \int_{\partial (B(\xi,r)\cap \Delta)}ds\leq 2\pi C_{1}r,
\end{align*}
where the constant $C_{1}$ depends only on the constant $C$ in the statement of the Theorem \ref{main}. Hence we have
$$\iint_{B(\xi, r)\cap \Delta}\displaystyle\frac{|\mu(w)|^{2}}{1-|w|^{2}}dudv\leq2\pi C_{1}r.\eqno(3.2)$$

 By the above discuss of the special case,
we easily get that
the measure $\displaystyle\frac{|\mu(z)|^{2}}{1-|z|^{2}}dxdy$ is a Carleson measure on $B_{i}\cap \Delta$ for any $1\leq i\leq n.$

Now we consider the general case. Let $G^{*}$ contains all the elements $g$ in $G$ such that $g(B)\cap B(\xi, r)\neq\emptyset.$ For $g\in G^{*}$ there are at most three possibilities as follows:

(a) there exist $1\leq i\leq n, $ $g(B_{i}\cap \mathcal{F})\subset B(\xi, r);$

(b) there exists $ 1\leq i\leq n, $ $g(B_{i}\cap \mathcal{F})\subset B(\xi, r)\neq\emptyset$ and $g(I_{i})\subset B(\xi, r)\cap\partial\Delta;$

(c)there exist $ 1\leq i\leq n, $ $g(B_{i}\cap \mathcal{F})\subset B(\xi, r)\neq\emptyset$ and $g(I_{i})\cap B(\xi, r)\cap\partial\Delta\neq\emptyset.$

For case (a),  we have

\begin{align*}
\iint_{g(B_{i}\cap \mathcal{F})}\displaystyle
 \frac{|\mu(w)|^{2}}{1-|w|^{2}}|dudv&\leq \iint_{g(B_{i}\cap \Delta)}\displaystyle
 \frac{|\mu(w)|^{2}}{1-|w|^{2}}|dudv\\
 &= \iint_{B_{i}\cap \Delta}\displaystyle
 \frac{|\mu(g(z))\displaystyle\frac{\overline{g'(z)}}
 {g'(z)}|^{2}}{1-|g(z)|^{2}}||g'(z)|^{2}dxdy\\
 &=\iint_{B_{i}\cap\Delta}\displaystyle
 \frac{|\mu(z)|^{2}}{1-|z|^{2}}||g'(z)|dxdy\\
 &\leq C_{1}\int_{\partial(B_{i}\cap\Delta)}|g'(z)|ds= C_{1}\int_{\partial g(B_{i}\cap\Delta)}ds\\
 &\leq C_{1}\pi length( g(B_{i}\cap\partial\Delta)),
\end{align*}
where the second inequality of above holds is by Lemma \ref{le2} and $C_{1}$ depend only on the Carleson norm of $\displaystyle
 \frac{|\mu(z)|^{2}}{1-|z|^{2}}|dxdy$ on $B_{i}\cap\Delta.$

For case (b) we have
\begin{align*}
&\iint_{g(B_{i}\cap \mathcal{F})\cap B(\xi,r)}\displaystyle
 \frac{|\mu(w)|^{2}}{1-|w|^{2}}|dudv\\
 &\leq\iint_{g(B_{i}\cap \Delta)\cap B(\xi,r)}\displaystyle
 \frac{|\mu(w)|^{2}}{1-|w|^{2}}|dudv\leq\pi C_{1}length(B_{i}\cap\partial\Delta).
\end{align*}
For case (c), notice that $g(B_{i}\cap\Delta)\cap B(\xi,r)$ is a triangle with three circle-arc and the angle corresponding to the side $g(B_{i}\cap\partial\Delta)\cap B(\xi,r)$ is bigger than some constant, we have

$$length(\partial(g(B_{i}\cap\Delta)\cap B(\xi,r)))\leq C_{2}length(g(B_{i}\cap\partial\Delta)\cap B(\xi,r)).$$
where the constant $C_{2}$ depends only on the Carleson norm of
$\frac{|\mu(z)|^{2}}{1-|z|^{2}}dxdy$ on $B_{i}\cap\Delta$ and the angle between $\partial B_{i}$ and $\partial\Delta$.

By the similar discuss as case(a) we have
$$\iint_{g(B_{i}\cap \mathcal{F})}\displaystyle
 \frac{|\mu(w)|^{2}}{1-|w|^{2}}dudv\leq \pi C_{2} length(g(B_{i}\cap\partial\Delta)\cap B(\xi,r)).$$

Since for every $1\leq i\leq n$, the arc $B_{i}\cap \partial\Delta$ does not contain the limit points of $G$. Hence for $g_{1}$, $g_{2}\in G^{*}$ if
$g_{1}(B_{i})\cap B(\xi, r)\neq\emptyset$ and $g_{2}(B_{i})\cap B(\xi, r)\neq\emptyset$, the images of $B_{i}\cap \partial\Delta$ under maps $g_{1}$, $g_{2}$, respectively, do not overlap.  Hence we have
\begin{align*}
&\iint_{B(\xi,r)\cap \Delta}\displaystyle
 \frac{|\mu(w)|^{2}}{1-|w|^{2}}dudv\leq \pi C^{*}\sum_{g\in G^{*}}length(g(B)\cap B(\xi,r)\cap \partial\Delta)\\&\leq \pi C^{*}length (B(\xi,r)\cap\partial\Delta)
\leq 2(\pi)^{2} C^{*}r,
\end{align*}
where $C^{*}$ equals to the maximum constants which appeared in the proof of this theorem and
$B=\bigcup_{i}^{n}(B_{i}\cap\Delta).$
Now the proof of the theorem is complete.
\end{proof}

 In\cite{Bi1}, Bishop showed that all divergence type Fuchsian groups hold Bowen's property, hence Theorem \ref{main} fails for the case of divergence-type groups. By Lemma \ref{le1} we know that for all convergence-type Fuchsian groups with compact support Beltrami coefficient, the discriminant method of Theorem \ref{main} also holds. It is natural to ask wether or not Theorem \ref{main} holds for all convergence-type Fuchsian groups.

\section{ Acknowledgements}

It is my pleasure to thank professor Michel Zinsmeister for inviting me to
the University of Orleans as a visiting scholar for one year and for some discussions on topics related to this paper. This work was done during my visiting Orleans. The author would also like to thank China Scholar Council for supporting my life in Orleans.


\begin{thebibliography}{99}
 \footnotesize

\bibitem {AZ}K. Astala and M. Zinsmeister.  {\em {\teich} spaces and BMOA.} Math. Ann., Vol 289(1991), 613-625.


\bibitem{AZ1} K. Astala and M. Zinsmeister. {\em Rectifiability in Teichm¨¹ller theory. in Topics in Complex Analysis,} Banach Center Publications Vol 31, (1995), 45-52.


\bibitem {Be} A. F. Beardon. {\em The geometry of discrete group.} Springer-Verlag, 1983.

\bibitem{Bi1} C.J. Bishop. {\em Divergence groups have the Bowen property.} Ann. Math., Vol 154(2001), 205-217.

\bibitem{Bi2} C.J. Bishop. {\em Compact deformations of Fuchsian group.} J. D'analyse Math., Vol 87(2002), 5-36.





\bibitem{Bo} R. Bowen. {\em Hausdorff dimension of quasicircles.} Publ. Math. IHES Vol50(1979), 11-25.

 \bibitem{Ca}L. Carleson. {\em An interpolation problem for bounded analytic functions}, Amer. J. Math., Vol 4 (1958), 921-930.

\bibitem{Cu}G. Cui, {\em Integrably asymptotic affine homeomorphisms of the circle and {\teich} spaces}, Sci. China Ser A, 43, (2000), 267-279.
\bibitem{Da}F. Dal'Bo. {\em Geodesic and horocyclic Trajectories}£¬ Springer, 2011.

\bibitem{Ga}J. B. Garnett: {\em  Bounded Analytic Functions,} New York: Academic Press,1981.

\bibitem{HW} S. Huo and S. Wu:  {\em The failure of analyticity of Hausdorff dimensions of quasi-circles of Fuchsian groups of the second kind,} Proc. Amer. Math. Soc. Vol 143(2015), 1101-1108.

\bibitem {HZ} S. Huo and M. Zinsmerster: {\em On Ruelle's property}. arXiv:1906.01291v1.
%



\bibitem{SW}Y. Shen and H. Wei, {\em Universal Teichmuller space and BMO}, Adv. Math, Vol 234(2013), 129-148.

%
%
%
%
%
%
%
\bibitem {Z} M. Zinsmeister: {\em Les domaines de Carleson}. Michigan Math. J., Vol 36, No 2, (1989), 213-220.
%
\end{thebibliography}
\end{document}